%% file: main_IEEE.tex
\documentclass{IEEEtran}

\usepackage{cite}
\usepackage{amsfonts}
\usepackage{mathrsfs}
\usepackage{graphicx}
\usepackage{textcomp}

\usepackage{xcolor}
\usepackage{booktabs}

\providecolor{backgroundblue}{rgb}{0.6,0.8,1}
\providecolor{backgroundorange}{rgb}{1,0.752,0}

\usepackage{amsthm}
\newtheorem{theorem}{Theorem}

\newtheorem{definition}{Definition}
\newtheorem{assumption}{Assumption}

\newtheorem{corollary}{Corollary}
\newtheorem{remark}{Remark}
\newtheorem{example}{Example}

\usepackage{amssymb} 
\usepackage{amsfonts}
\usepackage{mathtools}

\usepackage[bottom]{footmisc}
\usepackage{algorithmic}
\usepackage{algorithm}
\usepackage{comment}
\newboolean{showcomments}
\setboolean{showcomments}{false}

\newcommand{\fliu}[1]{\ifthenelse{\boolean{showcomments}}
        { \textcolor{red}{(FL:  #1)}}{}}
\newcommand{\xpeng}[1]{\ifthenelse{\boolean{showcomments}}
        { \textcolor{blue}{(XP:  #1)}}{}}
 \newcommand{\xru}[1]{\ifthenelse{\boolean{showcomments}}
        { \textcolor{violet}{(XR:  #1)}}{}}

\usepackage[colorlinks = true,
            linkcolor = cyan,
            urlcolor  = cyan,
            citecolor = cyan,
            anchorcolor = cyan]{hyperref}

\begin{document}

\title{
Positive Damping Region: A Graphic Tool for Passivization Analysis with Passivity Index
}
\author{
    Xiaoyu Peng,
    Xi Ru,
    Zhongze Li,
    Jianxin Zhang,
    Xinghua Chen,
    Feng Liu
\thanks{This paper is supported by the Smart Grid National Science and Technology Major Project under Grant 2024ZD0801200 and the Science and Technology Project of China Southern Power Grid Co., Ltd under Grant GDKJXM20250201(036000KC25020005). (\textit{Corresponding Author: Feng Liu}).}
\thanks{Xiaoyu Peng, Xi Ru, Zhongze Li, and Feng Liu are with the Department of Electrical Engineering, Tsinghua University, China.
Feng Liu is also with the State Key Laboratory of
Power System Operation and Control, Tsinghua University,
China
(email: pengxy19@tsinghua.org.cn; thu.ruxi@gmail.com; 
lizz24@mails.tsinghua.edu.cn; lfeng@tsinghua.edu.cn).}
\thanks{Jianxin Zhang is with the Power Dispatch Control Center, China Southern Power Grid Corporation, China (email: zhangjianxin@csg.cn).} 
\thanks{Xinghua Chen is with the Power Dispatch Control Center, Guangdong Power Grid Corporation, China (email: soulchen@126.com).}
}

\maketitle

\begin{abstract}
\input{abstract}
\end{abstract}


\input{text}

\bibliographystyle{IEEEtran}
\bibliography{reference}

\end{document}

%% file: abstract.tex
This paper develops a graphical framework that bridges passivization theory and engineering practice for linear time-invariant systems. We reveal that a system is passivizable with a given passivity index when the Nyquist plot
for SISO systems or the Rayleigh quotient of the transfer function for MIMO systems lies within a specific, index-dependent region
in the complex plane, termed the positive damping region. In particular, passivization can be checked from the Nyquist plot in SISO cases and from the Rayleigh quotient of the transfer function in MIMO cases. This viewpoint converts abstract passivity conditions into a visual tool for analyzing the feasibility of passivization, the associated frequency bands, the maximum achievable passivity index, and the trade-off among them. The framework can also be integrated with classical analysis and synthesis tools, including Nyquist and Nichols plots, as well as generalized KYP-based design. Applications of the proposed framework to passivity-based stability analysis and converter controller design in power systems illustrate its practical assessment and tuning capabilities.

%% file: text.tex
\section{Introduction}

Passivity theory is one of the cornerstones of stability analysis for dynamical systems \cite{Hill_Dissipativity_2022}. 
Owing to its modular structure, it is particularly valuable for analyzing large-scale interconnected systems, such as electrical power systems \cite{Schiffer_Conditions_2014, Peng_Compositional_2025}. 
With the development of output-feedback (OF) and input-feedforward (IF) passivization \cite{Hassan_Nonlinear_1996}, passivity theory has also become a framework for quantifying and reshaping non-passive systems. In this setting, passivity indices characterize both the distance to passivity and a subsystem's contribution to closed-loop stability.

A core condition of passivity is the positive damping property, which usually implies positive contributions to the stability of physical systems, and a larger positive passivity index indicates stronger damping. 
For instance, the passivity index is positively correlated with the equivalent admittance in power electronic circuits, which indicates better oscillation suppression capabilities \cite{Chen_Unified_2025, Feng_Unified_2025}.
At the same time, it is desirable to broaden the passive bandwidth even if the system cannot be passive at all frequencies.
Non-passivity situations across the entire frequency bands are common in practical power systems, such as the high-frequency bands of power inverters \cite{Huang_Gain_2024, Dey_PassivityBased_2023a}. 
Nevertheless, the pursuit of an expanded passive frequency bandwidth and an increased passivity index presents inherent contradictions. An excessively large passivity index results in very narrow, or even vanishing, passive bandwidths. 

Existing passivization studies primarily focus on maximizing the passivity index subject to full-frequency-bandwidth passivity constraints. For linear time-invariant (LTI) systems, a standard route is to convert passivity conditions into linear matrix inequalities (LMIs) via the Kalman-Yakubovich-Popov (KYP) lemma \cite{Kottenstette_relationships_2014}. 
Representative extensions account for additional constraints, such as frequency-dependent quadratic conditions \cite{Khong_Feedback_2025}, generalized passivity notions \cite{Yang_Distributed_2020, Bhowmick_LTI_2017}, and nonlinear settings \cite{DeS.Madeira_Necessary_2022}.
Another important direction is data-driven passivization: offline methods rely on iterative excitation and output measurements \cite{Tanemura_Efficient_2019}, whereas online methods use measured data directly \cite{Welikala_Online_2022}.

The methods above are usually expressed in algebraic or optimization-based forms and implicitly target the entire frequency spectrum. In practical interconnected systems, such as power systems, however, full-spectrum passivization may be unachievable through IF or OF structures \cite{Harnefors_PassivityBased_2016}. In such cases, a key engineering question is whether a system can be passivized over the desired frequency bands.
This issue is especially relevant in power electronics, where passivizing controllers are required to suppress oscillations on targeted frequencies \cite{Chen_Extended_2024}. 
An alternative approach is the generalized KYP lemma \cite{Iwasaki_Generalized_2005}, which provides an important tool for restricted-frequency analysis.
However, the resulting conditions of both remain less transparent for rapid assessment and interpretation, and the trade-off between the passivity index and the passivizable bandwidth is not intuitive.

This gap between rigorous theory and practical assessment motivates the present work. We represent passivization conditions geometrically in the complex plane, thereby converting abstract criteria into a directly interpretable graphical tool, which is more engineer-friendly.
This idea is primarily inspired by the circle criterion \cite{Hassan_Nonlinear_1996} and the disk margin theory \cite{Seiler_Introduction_2020}, but focuses on passivization rather than stability or robustness.
 Our main contribution is to provide a geometric characterization
of passivization conditions that enables direct visual analysis and control design. 

Specifically, we introduce positive damping (PD) conditions to determine whether an LTI system can be passivized for a prescribed passivity index. These conditions define an index-dependent region in the complex plane, called the positive damping region. A passivizable LTI system requires the Nyquist plot in the SISO case, or the Rayleigh quotient of the transfer function in the MIMO case, to lie within this region. This characterization provides a graphical tool for assessing the feasibility of IF and OF passivization, the associated passivizable frequency bands, the maximum achievable passivity index, and the trade-off among them. It also creates a direct interface between passivization analysis and practical design tools. We further demonstrate its use in passivity-based stability analysis and controller design for power systems. 

\textbf{Notation}: 
The real and complex number sets are $\mathbb{R}$ and $\mathbb{C}$.
$\Re(c)$ and $\Im(c)$ represent the real and imaginary parts of a complex number $c$.
The Kronecker product is denoted by $\otimes$.
The Rayleigh quotient of a matrix $M$ is $\rho_M(x) = x^H M x/(x^H x)$ for $x^H x\not=0$.
The numerical region for a matrix $M$ is $\mathcal{W}(M)=\{x^HMx: x\in\mathbb{C}, x^Hx=1\}$.
The normal and Hermite transpose are $M^T$ and $M^H$.
For a real symmetric matrix $M$, $\lambda_{\min}(M)$ denotes its minimum eigenvalue and $M\succ 0$ for its positive definiteness. 
A closed disk centering on $c\in\mathbb{C}$ with a positive radius $r\in\mathbb{R}$ is denoted by $\mathcal{D}(c,r)=\{z\in\mathbb{C}: |z-c|\leq r\}$, its boundary circle by $\partial \mathcal{D}(c,r)$, and its interior by ${\rm int}\mathcal{D}(c,r)=\mathcal{D}(c,r)\setminus\partial \mathcal{D}(c,r)$. 
The positive damping region and frequency bands w.r.t. a passivity index $\sigma$ defined in this paper are denoted by $\mathcal{P}_{\rm PD}(\sigma)$ and $\Omega_{\rm PD}(\sigma)$.

\textbf{Code}: All results in this paper can be reproduced by the code at \url{https://github.com/lingo01/Geometric_Passivization}.

\section{Preliminaries}
\begin{definition}[Passivity \cite{Hassan_Nonlinear_1996}]\label{def: passivity}
    Consider an LTI system with a minimal realization $\dot x=Ax+Bu, y=Cx+Du$.
    It is passive if its transfer function $H(s)=C(sI-A)^{-1}B+D$ satisfies the following conditions:
    \begin{enumerate}
        \item Poles of all elements of $H(s)$ are in the closed left-half plane.
        \item  For all $\omega\in\mathbb R$ for which $j\omega$ is not a pole of any element of $H(s)$, the matrix satisfies
        \begin{equation}\label{equ: PR condition 2}
            H(j\omega)+H^H(j\omega)\succeq 0
        \end{equation}
        \item  Any pure imaginary pole $j\omega$ of any element of $H(s)$ is a simple pole and the residue satisfies $\lim_{s\rightarrow j\omega}(s-j\omega)H(s)\succeq 0$.
    \end{enumerate}
    Besides, if $H(s-\varepsilon)$ is passive for a constant $\varepsilon>0$, it is called strictly passive.
\end{definition}

Among all three conditions in Def.~\ref{def: passivity}, the second one is of most significance.
For SISO systems, it reduces to $H(j\omega)+H^H(j\omega)=2\Re(H(j\omega))\geq 0$, i.e., the real part of its Nyquist plot stays positive, which usually implies positive damping in practice.
To emphasize its significance, we denote it as the \textit{positive damping (PD)} condition.

\begin{definition}[Positive Damping Condition]
A transfer function matrix $H(s)$ satisfies the PD condition at a frequency $\omega\in\mathbb R$ if $j\omega$ is not a pole of $H$ and satisfies \eqref{equ: PR condition 2}.

\end{definition}


\begin{figure}[htbp]
    \centering
    \includegraphics[width=3.in]{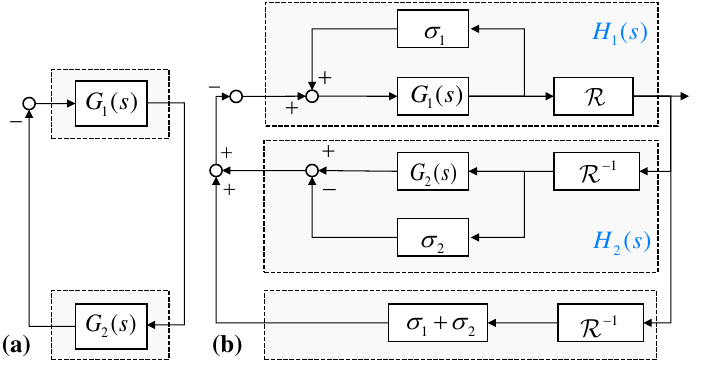}
    \caption{Feedback interconnected system model. (a): original model. (b): equivalent model with IF and OF passivization, where $\mathcal{R}=I$ for classical passivity definition. The operator is for compatibility with the latter discussions of generalized passivity definitions.}
 \label{figure: Passivity_Transformation}
\end{figure}
\fliu{We'd better clearly explain what is $\mathcal R$ in the caption of the Fig.~\ref{figure: Passivity_Transformation}}
\xpeng{Revised}

To broaden its applicability to feedback systems with non-passive subsystems in Fig.~\ref{figure: Passivity_Transformation}(a), the equivalent system in Fig.~\ref{figure: Passivity_Transformation}(b) derived from the loop transformation technique is usually considered. Then, if the OF system $H_1$ and the IF system $H_2$ are both passive and $\sigma_1 + \sigma_2 \geq 0$, the closed-loop system is passive \cite[Chapter 6]{Hassan_Nonlinear_1996}\fliu{When citing a book, we need to give the chapter and the number of the theorem.}\xpeng{Revised}. 
The formal definitions of IF and OF passivization are given as follows.

\begin{definition}[IF and OF Passivization]
For an LTI system $G(s)$ with a $\sigma\in\mathbb{R}$ as its passivity index, we say
\begin{enumerate}
    \item $G(s)$ is IF passive w.r.t. $\sigma$ if $H_\sigma(s) = G(s)-\sigma$ is passive.
    \item $G(s)$ is OF passive w.r.t. $\sigma$ if $H_\sigma(s) = (I - G(s)\sigma)^{-1}G(s)$ is passive and well-defined, i.e., $I-G(s)\sigma$ invertible.
\end{enumerate}

\end{definition}

For frequency-domain analysis, we also use the term OF-PD passive at a frequency $\omega$ to mean that $H_\sigma(j\omega)=(I-\sigma G(j\omega))^{-1}G(j\omega)$ is well-defined and satisfies the PD condition at $\omega$. This frequency-wise concept is weaker than full OF passivity unless the other two conditions in Def.~\ref{def: passivity} are also verified.

The case of IF passivization is relatively simple to analyze by noting the linear correlation between $H$ and $G$. Therefore, this paper investigates OF passivization and extends it to IF passivization.

\section{Visualizing Passivization of MIMO Systems}
\subsection{Positive Damping Condition}

Throughout the paper, we consider transfer functions that satisfy the following condition.
\begin{assumption}\label{assump: transfer function}
    A transfer function is real and strictly proper.
\end{assumption}

{\color{black}
\begin{theorem}[MIMO OF PD Condition]\label{theo: MIMO}
Let $G(s)$ be a square transfer-function matrix satisfying Assum.~\ref{assump: transfer function}. Consider a scalar passivity index $\sigma\in\mathbb R$ and a frequency $\omega\in\mathbb R$ such that $j\omega$ is not a pole of $G$ and $I-\sigma G(j\omega)$ is nonsingular. Define
\begin{equation}
    H_\sigma(j\omega)=(I-\sigma G(j\omega))^{-1}G(j\omega).
\end{equation}
Then the following statements hold.
\begin{enumerate}
    \item $H_\sigma$ satisfies the PD condition \eqref{equ: PR condition 2} at $\omega$ if and only if
    \begin{equation}\label{equ: MIMO LMI expression}
    G(j\omega)+G^H(j\omega)-2\sigma G(j\omega)G^H(j\omega)\succeq0 .
    \end{equation}
    \item If $\sigma\geq0$, a necessary condition for \eqref{equ: MIMO LMI expression} is
    \begin{equation}\label{equ: PR condition MIMO}
    \Re(\rho_{G(j\omega)}(x))\geq \sigma |\rho_{G(j\omega)}(x)|^2,
    \quad \forall x: x^Hx\neq0 .
    \end{equation}
    \item If $G(j\omega)$ is nonsingular, then \eqref{equ: MIMO LMI expression} is equivalent to
    \begin{equation}\label{equ: MIMO rayleigh equivalent}
        \Re(\rho_{G^{-1}(j\omega)}(x))\geq\sigma,
        \quad \forall x: x^Hx\neq0 .
    \end{equation}
    \item If $G^{-1}(j\omega)$ is nonsingular and sectorial \cite{Chen_Phase_2024}, i.e., $\rho_{G^{-1}(j\omega)}(x)\neq0$, then \eqref{equ: MIMO LMI expression} is equivalent to
   \begin{equation}\label{equ: MIMO rayleigh equivalent-inverse}
   \begin{aligned}
   \Re\left((\rho_{G^{-1}(j\omega)}(x))^{-1}\right) \geq \sigma \left|{(\rho_{G^{-1}(j\omega)}(x))^{-1}}\right|^2,\quad \\
       \forall x: x^Hx\neq0 .
   \end{aligned}
   \end{equation}
\end{enumerate}
\end{theorem}

\begin{proof}
For compactness, write $G=G(j\omega)$ and $H=H_\sigma(j\omega)$. Since $I-\sigma G$ is nonsingular,
\begin{equation}
H+H^H=(I-\sigma G)^{-1}G+G^H(I-\sigma G^H)^{-1}.
\end{equation}
Applying the congruence transformation with $I-\sigma G$ gives
\begin{equation}
(I-\sigma G)(H+H^H)(I-\sigma G)^H=G+G^H-2\sigma GG^H.
\end{equation}
The complex congruent transformation preserves positive semi-definiteness, thereby proving the first statement.

For the second statement, assume $\sigma\geq0$ and \eqref{equ: MIMO LMI expression} holds. For any $x\neq0$,
\begin{equation}
    2\Re(x^HGx)\geq2\sigma x^HGG^Hx.
\end{equation}
By the Cauchy-Schwarz inequality,
\begin{equation}
    x^HGG^Hx\geq \frac{|x^HGx|^2}{x^Hx}=|\rho_G(x)|^2x^Hx .
\end{equation}
Because $\sigma\geq0$, multiplying this inequality by $\sigma$ preserves the inequality direction. Hence
\begin{equation}
    \Re(\rho_G(x))\geq\sigma |\rho_G(x)|^2,
\end{equation}
which proves the necessary condition. 

If $G$ is nonsingular, pre- and post-multiplying \eqref{equ: MIMO LMI expression} by $G^{-1}$ and $G^{-H}$ gives $G^{-1}+G^{-H}-2\sigma I\succeq0$, which is equivalent to the Rayleigh-quotient inequality \eqref{equ: MIMO rayleigh equivalent}.

Finally, the fourth statement follows from combining \eqref{equ: MIMO rayleigh equivalent} and the following identity
\begin{equation}
    |\rho_{G^{-1}}(x)|^2\Re\left((\rho_{G^{-1}}(x))^{-1}\right)=\Re(\rho_{G^{-1}}(x)).
\end{equation}
\end{proof}
}

{\color{black}
The Rayleigh-quotient screening condition in Thm.~\ref{theo: MIMO} is closely associated with the numerical region \cite{Chen_Phase_2024}. 
For example, the condition \eqref{equ: MIMO rayleigh equivalent} in the theorem can be rewritten as $\mathcal W(G^{-1}(j\omega))\subseteq\{z\in\mathbb C:\Re (z)\geq\sigma\}$.
Moreover, it leads to the following gain-phase diagnostic.}
{\color{black}
\begin{corollary}[Gain-Phase MIMO OF-PD Screening]\label{corollary: MIMO gain-phase}
Let $G(s)$ be a transfer-function matrix satisfying Assum.~\ref{assump: transfer function}. Suppose $\sigma\geq0$ and $\rho_{G(j\omega)}(x)\neq0$ for the Rayleigh quotients under consideration. A necessary condition for the MIMO OF-PD condition at $\omega$ is
    \begin{equation}
        \sigma\leq \inf_{z\in\mathcal{W}(G(j\omega)),\, z\neq0}\frac{\cos(\angle z)}{|z|}.
    \end{equation}
\end{corollary}
\begin{proof}
The result follows by writing \eqref{equ: PR condition MIMO} as $\sigma\leq \Re (z)/|z|^2=\cos(\angle z)/|z|$ for every nonzero Rayleigh quotient $z$.
\end{proof}
This result establishes a close relation between gain, phase, and passivization capability.
Besides large gains, stronger cross-channel couplings may enlarge the phase spread of the numerical region, thereby reducing $\cos(\angle z)$ and rapidly degrading the passivization capability.
It explains why a large matrix phase defined as $\max_{z\in\mathcal{W}(G(j\omega))}\arg(z)$ in power converters leads to coupled instability \cite{Huang_Gain_2024, Feng_Unified_2025}.
Moreover, the third result of Thm.~\ref{theo: MIMO} provides a clear geometry, i.e., the numerical region of $G^{-1}(j\omega)$ should stay on the right of $\sigma$, indicating this fact from an alternative perspective.}

\subsection{Visualization of PD Region and Its Property}
{\color{black}
The exact MIMO condition \eqref{equ: MIMO LMI expression} is necessary and sufficient, and the Rayleigh quotient condition \eqref{equ: MIMO rayleigh equivalent} provides an exact half-plane geometry when $G(j\omega)$ is nonsingular. The Rayleigh condition \eqref{equ: PR condition MIMO} of $G(j\omega)$, in contrast, gives a quick check.
Specifically, \eqref{equ: MIMO rayleigh equivalent} defines the output-feedback positive-damping region
\begin{equation}\label{equ: comp 4}
\mathcal P_{\rm PD}(\sigma)=
\begin{cases}
\mathcal D\left(\dfrac{1}{2\sigma}+j0,\left|\dfrac{1}{2\sigma}\right|\right), & \sigma>0,\\[1ex]
\{z\in\mathbb C:\Re (z)\geq0\}, & \sigma=0,\\[1ex]
\mathbb C\setminus{\rm int}\,\mathcal D\left(\dfrac{1}{2\sigma}+j0,\left|\dfrac{1}{2\sigma}\right|\right), & \sigma<0.
\end{cases}
\end{equation}
For $\sigma>0$, this region is the closed disk centered at $(1/(2\sigma),j0)$ with radius $1/(2\sigma)$. For $\sigma<0$, it is the exterior of the corresponding disk with the nonnegative radius $|1/(2\sigma)|$. Inclusion of all Rayleigh quotients of $G(j\omega)$ in this disk is necessary but does not by itself guarantee \eqref{equ: MIMO LMI expression}.}

Then, the frequency bands that fulfill the PD condition, referred to as the \emph{PD frequency band}, can be estimated from this interpretation. Furthermore, it provides a geometric explanation for why increasing the OF passivity index yields a narrower PD bandwidth and a more stringent passivity condition, as summarized in the subsequent corollary.

{\color{black}
\begin{definition}[OF-PD Frequency Band]
Consider a transfer-function matrix $G(s)$ satisfying Assum.~\ref{assump: transfer function} and a scalar index $\sigma$. Define $H_\sigma(s)=(I-\sigma G(s))^{-1}G(s)$ wherever this transfer matrix is well-defined. The output-feedback positive-damping frequency band w.r.t. $\sigma$ is
\begin{equation}
\begin{aligned}
\Omega_{\rm PD}(\sigma)=\{\omega\geq0:
& j\omega\text{ is not a pole of }G,\\ 
& I-\sigma G(j\omega)\text{ is nonsingular}, \\
&H_\sigma(j\omega)+H_\sigma^H(j\omega)\succeq0\}.
\end{aligned}
\end{equation}
For $\sigma\geq0$, define the Rayleigh-screened band
\begin{equation}
\tilde\Omega_{\rm PD}(\sigma)=\{\omega\geq0:\eqref{equ: PR condition MIMO}\text{ holds for all }x\neq0\}.
\end{equation}
In general MIMO cases, $\tilde\Omega_{\rm PD}(\sigma)$ is only a necessary screening band.
\end{definition}

\begin{corollary}[OF-PD Bandwidth Contraction]\label{corollary: contraction}
Consider a transfer-function matrix $G(s)$ satisfying Assum.~\ref{assump: transfer function}. For any real indices $\sigma_2>\sigma_1$, the exact MIMO OF-PD condition is monotone in the following sense: at every frequency where both $I-\sigma_1G(j\omega)$ and $I-\sigma_2G(j\omega)$ are nonsingular,
\begin{equation}
\omega\in\Omega_{\rm PD}(\sigma_2)\quad\Longrightarrow\quad \omega\in\Omega_{\rm PD}(\sigma_1).
\end{equation}
Moreover, for $\sigma_2>\sigma_1\geq0$, the Rayleigh-screened bands satisfy
\begin{equation}
\tilde\Omega_{\rm PD}(\sigma_2)\subseteq\tilde\Omega_{\rm PD}(\sigma_1).
\end{equation}
\end{corollary}

\begin{proof}
Let $A_\sigma(j\omega)=G(j\omega)+G^H(j\omega)-2\sigma G(j\omega)G^H(j\omega)$.
If $\sigma_2>\sigma_1$, then
\begin{equation}
A_{\sigma_1}(j\omega)=A_{\sigma_2}(j\omega)+2(\sigma_2-\sigma_1)G(j\omega)G^H(j\omega)\succeq A_{\sigma_2}(j\omega).
\end{equation}
Thus $A_{\sigma_2}\succeq0$ implies $A_{\sigma_1}\succeq0$, proving the exact-band contraction on the common regularity domain. The screened-band inclusion follows directly from $\Re\rho_G\geq\sigma_2|\rho_G|^2$ implying $\Re\rho_G\geq\sigma_1|\rho_G|^2$ when $\sigma_2>\sigma_1\geq0$.
\end{proof}}

\subsection{Dual Results for Input-Feedforward Passivization}

The proposed method also applies to visualizing the IF passivization $H = G - \sigma I$. Due to the parallel structure of both $G$ and $\sigma$, it is more trivial than the OF counterparts, where the PD condition can be equivalently written as $G+G^H-2\sigma I \succeq 0$.
Similarly, if employing the Rayleigh quotient, we can estimate the IF PD region as $\{\rho_G(x): \Re(\rho_G(x)) \geq \sigma, \forall x: x^Hx\not=0\}$.
Following the ideas of OF passivization, we can visualize the PD region as a vertical half-plane and derive the property of PD frequency-band contraction.


\section{Visualizing Passivization of SISO Systems}
\fliu{Intuitively, the SISO case should be a special case of MIMO? So why do we need to present this section specifically? If not, why don't we place the SISO case before the MIMO one?}\xpeng{There are more elegant results for SISO systems, which are hardly extended to MIMO ones.}

Passivity theory decomposes networked systems into devices and their interconnections. In many engineering settings, device models are scalar either directly or after suitable transformations \cite{Zhou_SmallSignal_2023}. 
Therefore, the SISO case plays a particularly important role in bridging theoretical passivization conditions and practical frequency-domain assessment.

\subsection{Positive Damping Condition}

{\color{black}
\begin{theorem}[SISO OF-PD Condition]\label{theo: SISO}
Let $G(s)$ be a scalar transfer function satisfying Assum.~\ref{assump: transfer function}. For an index $\sigma\in\mathbb R$ and a frequency $\omega\in\mathbb R$ where $j\omega$ is not a pole of $G$ and $1-\sigma G(j\omega)\neq0$, the transformed system $H_\sigma(s)=G(s)/(1-\sigma G(s))$ satisfies the PD condition \eqref{equ: PR condition 2} at $\omega$ if and only if
\begin{equation}\label{equ: PR condition SISO}
\Re(G(j\omega)) \geq \sigma |G(j\omega)|^2 .
\end{equation}
\end{theorem}
\begin{proof}
Substituting $H_\sigma(j\omega)=G(j\omega)/(1-\sigma G(j\omega))$ gives
\begin{equation}
\Re(H_\sigma(j\omega))=\frac{\Re(G(j\omega))-\sigma|G(j\omega)|^2}{|1-\sigma G(j\omega)|^2}.
\end{equation}
The denominator is strictly positive under the well-posedness assumption $1-\sigma G(j\omega)\neq0$. Since $H_\sigma+H_\sigma^H=2\Re(H_\sigma)$ in the SISO case, the PD condition is equivalent to \eqref{equ: PR condition SISO}.
\end{proof}}

\subsection{Visualization of PD Region and Its Property}

For SISO systems, the OF-PD region satisfying \eqref{equ: PR condition SISO} is exactly $G(j\omega)\in\mathcal P_{\rm PD}(\sigma)$ with $\mathcal P_{\rm PD}(\sigma)$ defined in \eqref{equ: comp 4}. Hence, for $\sigma>0$, the Nyquist plot of $G(j\omega)$ must lie inside the disk centered at $(1/(2\sigma),j0)$ with radius $1/(2\sigma)$; for $\sigma<0$, it must lie outside the disk with the same center and radius $|1/(2\sigma)|$; and for $\sigma=0$, it must lie in the closed right-half plane.

At the level of the PD condition, the geometry of OF passivization can be alternatively understood from the M{\"o}bius transformation $H_\sigma=G/(1-\sigma G)$. The right-half plane in the $H_\sigma$-plane is mapped back to the disk/exterior/half-plane region in the $G$-plane. Full OF passivity still requires the pole-location and residue conditions in Def.~\ref{def: passivity}; these are addressed by the sufficient geometric certificate below. Similarly, IF passivization can be visualized since it also implies the affine transformation $H=G-\sigma$.

\begin{remark}
This M{\"o}bius perspective suggests the proposed method also holds for discrete systems derived from the bilinear method. Consider Tustin's mapping $z=(1+sT/2)/(1-sT/2)$, which is also a M{\"o}bius transformation.
Thus, the circle-preservation property guarantees that the PD region will be preserved as a disk in the $z$-domain.
\end{remark}

Similarly to MIMO systems, we can also derive the PD bandwidth contraction property for SISO systems.
Due to the necessity and sufficiency of Thm.~\ref{theo: SISO}, the result is stronger as shown below:

{\color{black}
\begin{corollary}[SISO OF-PD Bandwidth Contraction]
Consider a scalar transfer function $G(s)$ satisfying Assum.~\ref{assump: transfer function}. Then the Rayleigh-screened and exact OF-PD bands coincide, $\tilde\Omega_{\rm PD}(\sigma)=\Omega_{\rm PD}(\sigma)$, whenever both are defined. Moreover, for any indices $\sigma_2>\sigma_1$, the inclusion
\begin{equation}
\Omega_{\rm PD}(\sigma_2)\subseteq\Omega_{\rm PD}(\sigma_1)
\end{equation}
holds on the common regularity domain where both $1-\sigma_iG(j\omega)\neq0$, $i=1,2$.
\end{corollary}}

{\color{black}
\begin{proof}
The equality of the screened and exact bands follows from the necessary and sufficient scalar condition in Thm.~\ref{theo: SISO}. For a fixed frequency, if $\Re( G(j\omega))\geq\sigma_2|G(j\omega)|^2$ and $\sigma_2>\sigma_1$, then $\Re( G(j\omega))\geq\sigma_1|G(j\omega)|^2$, which completes the proof.
\end{proof}}

The results are intuitive from the graphical analysis shown in Fig.~\ref{figure: Illu_1}: compared to classical passivity ($\sigma=0$), the OF passivization w.r.t. $\sigma>0$ is more stringent. In contrast, $\sigma<0$ allows certain systems that were previously non-passive to be rendered OF passive.

\begin{figure}[htbp]
    \centering
    \includegraphics[width=3.in]{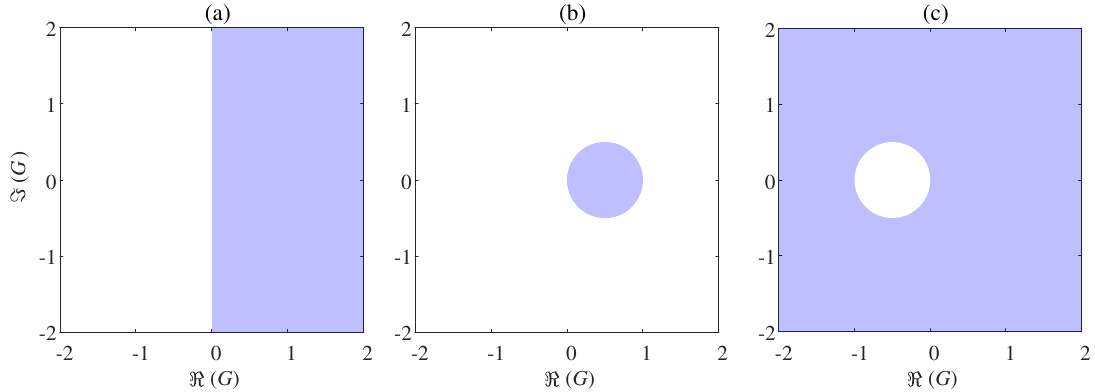}
    \caption{Illustration of PD regions $\mathcal{P}_{\rm PD}(\sigma)$. (a): $\sigma=$0. 
    (b): $\sigma=$1. (c): $\sigma=$-1. 
}
 \label{figure: Illu_1}
\end{figure}

Inspired by the above interpretation, we can derive a geometric condition of OF passivization as follows.

{\color{black}
\begin{theorem}[Geometric OF Passivity Certificate]\label{theo: SISO passivity condition}
Let $G(s)$ be a scalar rational transfer function satisfying Assum.~\ref{assump: transfer function} with no poles on the imaginary axis, and let $\sigma\neq0$. Define $F(s)=1-\sigma G(s)$ and $H_\sigma(s)=G(s)/F(s)$. Then $G(s)$ is OF passive w.r.t. $\sigma$ if the following conditions hold:
\begin{enumerate}
    \item Under the Nyquist convention used here, the counter-clockwise winding number made by the Nyquist plot of $G(s)$ around the point $(1/\sigma,j0)$ equals the number of unstable poles of $G(s)$.
    \item For every $\omega$ for which $G(j\omega)$ is finite and $1-\sigma G(j\omega)\neq0$, the Nyquist plot satisfies $G(j\omega)\in\mathcal P_{\rm PD}(\sigma)$.
    \item If the Nyquist plot intersects $(1/\sigma,j0)$ for each positive crossing frequency $\omega_c>0$, the intersection is simple and vertical from the upper to the lower half plane, i.e.,
        \begin{equation}\label{equ: G intersect condition}
            \frac{{\rm d}}{{\rm d}\omega}\Re (G(j\omega_c))=0,\quad
            \frac{{\rm d}}{{\rm d}\omega}\Im (G(j\omega_c))<0,
        \end{equation}
        and the root $s=j\omega_c$ of $F(s)=0$ is simple.
\end{enumerate}
\end{theorem}

\begin{proof}
The poles of $H_\sigma$ are the poles of $G$ not canceled by $F$, together with the zeros of $F(s)=1-\sigma G(s)$. Since $G$ has no imaginary-axis poles, the Nyquist/argument-principle condition in item 1) gives that $F$ has the required number of zeros so that no right-half-plane pole is introduced into $H_\sigma$. Item 2) and Thm.~\ref{theo: SISO} give the PD condition in Def.~\ref{def: passivity} wherever $H_\sigma(j\omega)$ is finite.

It remains to check imaginary-axis poles generated by intersections with $G(j\omega)=1/\sigma$. If no such intersection exists, then $H_\sigma(s)=G(s)/(1-\sigma G(s))$ has no imaginary-axis pole. If $F(j\omega_c)=0$ and the root is simple, the residue is
\begin{equation}
\begin{aligned}
{\rm Res}(H_\sigma,j\omega_c)&=\lim_{s\to j\omega_c}\frac{(s-j\omega_c)G(s)}{1-\sigma G(s)}\\
&=-\frac{1}{\sigma^2G'(j\omega_c)}
=\frac{-j}{\sigma^2\frac{{\rm d}G(j\omega_c)}{{\rm d}\omega}} .
\end{aligned}
\end{equation}
Writing $G(j\omega)=u(\omega)+jv(\omega)$ gives
\begin{equation}
{\rm Res}(H_\sigma,j\omega_c)=
\frac{-v'(\omega_c)-ju'(\omega_c)}{\sigma^2|G'(j\omega_c)|^2}.
\end{equation}
The residue condition in Def.~\ref{def: passivity} is therefore satisfied when $u'(\omega_c)=0$ and $v'(\omega_c)<0$, which is the condition 3). Hence, the three passivity conditions in Def.~\ref{def: passivity} hold.
\end{proof}}

\subsection{Robustness of Passivization}
The robustness of passivization can be assessed by translating passivization margins into intuitive geometric distances and regions. To demonstrate its strengths, undirected and polytopic uncertainties are considered.

\textbf{1) Undirected Uncertainty}:

Consider a nominal system $G_0(s)$ with an undirected and additive perturbation $\Delta(s)$, i.e., $G=G_0+\Delta$. The robustness of its OF-PD condition is investigated w.r.t. an index $\sigma>0$. The analysis relies on the minimum distance of Nyquist plots to the boundary of PD regions defined as
\begin{equation}\label{equ: robustness distance}
    d(\sigma, \omega)= {1}/{(2\sigma)} - \left| G_0(j\omega) - {1}/{(2\sigma)} \right| .
\end{equation}
For a nominal system whose Nyquist plot lies strictly inside $\mathcal P_{\rm PD}(\sigma)$, $d(\sigma,\omega)>0$ on the considered frequency set.
\begin{corollary}[Undirected Robust OF-PD Condition]\label{theo: robustness}
Let $\sigma>0$ and let $G_0(s)$ satisfy the OF-PD condition on a frequency set $\Omega$ with distance $d(\sigma,\omega)>0$. Consider $G(s)=G_0(s)+\Delta(s)$. If $|\Delta(j\omega)|<d(\sigma,\omega)$ for all $\omega\in\Omega$ and $1-\sigma G(j\omega)\neq0$ on $\Omega$, then $G(s)$ also satisfies the OF-PD condition on $\Omega$. If, in addition, the pole-location and residue requirements in Def.~\ref{def: passivity} are verified by the Nyquist certificate in Thm.~\ref{theo: SISO passivity condition}, then the corresponding full OF passivity conclusion follows.
\end{corollary}
\begin{proof}
Denote $\mathcal{P}_{\rm PD}(\sigma)=\mathcal{D}(c,r)$ with $c=r=1/(2\sigma)$. The triangle inequality gives
\begin{equation}
|G(j\omega)-c|\leq |G_0(j\omega)-c|+|\Delta(j\omega)|<r.
\end{equation}
Thus $G(j\omega)\in\mathcal D(c,r)$, and Thm.~\ref{theo: SISO} gives the OF-PD condition on $\Omega$. Full passivity, when claimed, requires the additional pole and residue checks stated in the corollary.
\end{proof}

\textbf{2) Polytopic Uncertainty}:

In industrial applications, state matrices are influenced by parameter uncertainties, such as inaccurate resistance estimation and time-varying grid strength in power systems.
Then, the perturbed system can be described by an affine combination of different nominal systems.
By using the proposed method, the passivization robustness under polytopic uncertainty can be transformed into an LMI problem.
\begin{corollary}[Polytopic Robust PD Inclusion]
Let $G(s)$ be a system with minimal realization $(A,B,C,D)$. Suppose it has polytopic uncertainty written as
\begin{equation}\label{equ: affine ABCD}
    (A,B,C,D)= \sum_{i=1}^n\left( \alpha_iA_i, \alpha_iB_i,  \alpha_iC_i, \alpha_iD_i\right)
\end{equation}
where the affine parameters $\alpha_i\geq0$, $\sum_i\alpha_i=1$.
Then, $G(s)$ is OF passive w.r.t. the index $\sigma$ if there exist $P_i\succ 0$ for each $i=1,2,\cdots, n$ and a common matrix $F$ independent from $\alpha_i$ such that
\begin{equation}\label{equ: LMI vertex}
    \Theta_i + F N_i + N_i^T F^T \prec 0,\; i=1,2,\cdots,n
\end{equation}
where 
\begin{equation}
    \Theta_i=\begin{bmatrix}
        0 & P_i &  & \\
        P_i & 0 &  & \\
         & &  I &-\frac{I}{2\sigma}\\
         & & -\frac{I}{2\sigma} & 0
    \end{bmatrix},
    N_i=
        \begin{bmatrix} A_i & -I & 0 & B_i \\ C_i & 0 & -I & D_i \end{bmatrix}
\end{equation}
\end{corollary}

\begin{proof}
    First, the PD disk inclusion condition $G(j\omega)\in\mathcal{D}(1/(2\sigma),1/(2\sigma))$ can be equivalently expressed as
    \begin{equation}\label{equ: passivity LMI}
        \begin{bmatrix}G(j\omega)\\ I\end{bmatrix}^H\Pi
        \begin{bmatrix}G(j\omega)\\ I\end{bmatrix}\leq 0,\,
        \Pi=\begin{bmatrix}
            \Pi_{11} & \Pi_{12}\\
            \Pi_{21} & \Pi_{22}\end{bmatrix}=\begin{bmatrix}
            I & -\frac{I}{2\sigma}\\
            -\frac{I}{2\sigma} & 0
        \end{bmatrix}
    \end{equation}
    The LMI \eqref{equ: passivity LMI} has the equivalent expression written as
    \begin{equation}
        \begin{bmatrix} (j\omega I - A)^{-1}B \\ I \end{bmatrix}^H \begin{bmatrix} C & D \\ 0 & I \end{bmatrix}^T \Pi \begin{bmatrix} C & D \\ 0 & I \end{bmatrix}\begin{bmatrix} (j\omega I - A)^{-1}B \\ I \end{bmatrix} \preceq 0
    \end{equation}

    Under the standard strict KYP assumptions for the considered realization, the frequency-domain inequality is certified by the existence of a matrix $P\succ 0$ such that
    \begin{equation}
        \Gamma^T \Theta \Gamma\prec0,\,
        \Theta=\begin{bmatrix} 0 & P &  &  \\ P & 0 &  &  \\  &  & I & -\frac{I}{2\sigma} \\  &  & -\frac{I}{2\sigma} & 0 \end{bmatrix},\,
        \Gamma=\begin{bmatrix} I & 0 \\ A & B \\ C & D \\ 0 & I \end{bmatrix}
    \end{equation}
    In other words, $\eta^T \Gamma^T \Theta \Gamma \eta < 0$, $\forall\eta\not=0$. Let $\eta=[x^T, u^T]^T$, then the vector $\Gamma \eta$ can be computed as
    \begin{equation}
        \Gamma \eta=\begin{bmatrix} I & 0 \\ A & B \\ C & D \\ 0 & I \end{bmatrix}\begin{bmatrix}
            x\\u
        \end{bmatrix} = \begin{bmatrix} x \\ Ax+Bu \\ Cx+Du \\ u \end{bmatrix} = \begin{bmatrix} x \\ \dot{x} \\ y \\ u \end{bmatrix}
    \end{equation}
    Define $\zeta=\Gamma\eta=[x^T,\, \dot x^T,\, y^T,\, u^T]^T$.
    Then, the state expressions $\dot x=Ax+Bu.\, y=Cx+Du$ can be reformulated as $N\zeta=0$, where
    \begin{equation}
        N=\begin{bmatrix} A & -I & 0 & B \\ C & 0 & -I & D \end{bmatrix}
    \end{equation}
    Then, the original requirement $\eta^T \Gamma^T \Theta \Gamma \eta < 0, \forall \eta \ne 0$ can be transformed into $\zeta^T \Theta \zeta < 0, \forall \zeta \ne 0$ subject to $N\zeta = 0$. Then, by employing the Finsler's Lemma \cite{de2007stability}, it is equivalent to find a matrix $F$ such that
    \begin{equation}
        \Theta + F N+ N^T F^T \prec 0
    \end{equation}
    Let $P=\sum_i \alpha_iP_i$ and $N=\sum_i \alpha_iN_i$. Substituting them and \eqref{equ: affine ABCD} into the inequality yields
    \begin{equation}
        \sum_{i=1}^n \alpha_i\left(\Theta_i + F N_i + N_i^T F^T\right)\prec0,\, \sum_{i=1}^n\alpha_i=1
    \end{equation}
   Equivalently, it requires the inequality to hold on each vertex \eqref{equ: LMI vertex}, which completes the proof.
\end{proof}

\subsection{Waterbed Effect of Passivization}

In SISO systems, Thm.~\ref{theo: SISO} suggests that the following $\varsigma(\omega)$ serves as a metric of OF passivizable ability at the frequency $\omega$ where $G(j\omega)\not=0$:
\begin{equation}\label{equ: comp 8}
    {\varsigma}(\omega) = {\Re(G(j\omega))}/{|G(j\omega)|^2} = \Re\left({G^{-1}(j\omega)}\right)
\end{equation}
A larger index $\varsigma(\omega)$ quantifies the system's stronger ability to be OF passivated and thus indicates a better damping effect at the frequency $\omega$.
It can be viewed as a generalized and frequency-wise OF passivity index and satisfies $\sigma\leq\varsigma(\omega)$ as delineated in \eqref{equ: PR condition SISO}.
However, we show $\varsigma(\omega)$ cannot be strengthened at all frequencies for a specific system but follows a conservation law, akin to the \textit{waterbed effect} \cite{Seron_Fundamental_1997}. 

Specifically, we consider a strictly proper rational, open-loop stable, and minimum-phase LTI system with no zeros on the imaginary axis $G(s)$. Then, it can be reformulated as ${1}/{G(s)}=L(s)+C+R(s)$, where $L(s)$, $C$, and $R(s)$ are the polynomials without the constant term, constant terms, and strict rational functions, respectively.
The order of $L(s)$ equals the relative degree $n_r$ of $G(s)$.
Consider the rational part denoted by $f(s)$, written as $f(s) = {1}/{G(s)}-L(s)=C+R(s)$.
The minimum-phase assumption guarantees that $f(s)$ is analytic on the closed right-half plane and $\lim_{D\to\infty}\sup_\theta |f(De^{j\theta})|/D=0$. Therefore, $\Re(f(s))$ is harmonic and follows the Poisson formula \cite{Seron_Fundamental_1997}
\begin{equation}
    \pi\Re(f(a)) = \int_{-\infty}^{\infty} \Re(f(j\omega))\mathscr{P}_a(\omega){\rm d}\omega
\end{equation}
{\color{black}where $a\in\mathbb R^+$ and $\mathscr{P}_a(\omega)=a/(\omega^2+a^2)$ is the Poisson kernel.} 
The constant $a$ regulates the shapes of the kernel, and a smaller $a$ indicates more concern for the low-frequency bands.
For a given system, $a$ should be predetermined as a constant.
Substituting $f(s)$ yields
\begin{equation}\label{equ: comp 10}
    \frac{1}{G(a)}-L(a)=\frac{1}{\pi}\int_{-\infty}^{\infty} \left[\varsigma(\omega)-\Re(L(j\omega))\right]\mathscr{P}_a(\omega){\rm d}\omega
\end{equation}
 
For high relative degrees $n_r>1$, one obtains additional polynomial subtraction terms in the right-hand side of \eqref{equ: comp 10}, but the qualitative conservation persists.
For $n_r=1$, we can write $L(s)=l_1s$ where $l_1=1/(\lim_{s\to\infty}sG(s))$ is a real coefficient and \eqref{equ: comp 10} reduces to
\begin{equation}\label{equ: comp 11}
    \frac{1}{G(a)}-l_1a = \frac{1}{\pi}\int_{-\infty}^{\infty} \varsigma(\omega)\mathscr{P}_a(\omega){\rm d}\omega
\end{equation}
which implies that the (Poisson-weighted) integral of $\varsigma(\omega)$ over the entire frequency bands is conservative for a given system $G(s)$.
Therefore, increased damping in some bands (larger $\varsigma$) will force reductions elsewhere.
For instance, suppose we hope $\varsigma(\omega)\geq \sigma$ for $|\omega|\leq \omega_c$ and $\varsigma(\omega)\geq 0$ for $|\omega|> \omega_c$. Then, \eqref{equ: comp 11} with $a>0$ yields
\begin{equation}
        \frac{1}{G(a)}-l_1a \geq \frac{1}{\pi}\int_{-\omega_c}^{\omega_c} \varsigma(\omega)\mathscr{P}_a(\omega){\rm d}\omega\geq \frac{\sigma}{\pi}\int_{-\omega_c}^{\omega_c} \mathscr{P}_a(\omega){\rm d}\omega\\
\end{equation}
Direct computation leads to ${1}/{G(a)}-l_1a\geq{\sigma}/{\pi}\arctan\left({\omega_c}/{a}\right)$.
This inequality clearly shows that the passivizable frequency bandwidth $\omega_c$ and the passivity index $\sigma$ cannot be simultaneously extended for a predetermined transfer function $G(s)$.


\section{Extension to Generalized Definitions}\label{subsec: extension to generalized passivity definition}

Apart from the classical passivity definition setting $\mathcal{R}=I$ in Fig.~\ref{figure: Passivity_Transformation}(b), some generalized definitions are also employed to extend the applicability of the passivity theory, which can also be contained in our framework.


\begin{definition}[$\mathcal{R}$-OF Passivization]
    An LTI system $G(s)$ is $\mathcal R$-OF passivizable with respect to a rational operator $\mathcal R(s)$ and an index $\sigma\in\mathbb R$ if the frequency response of $\mathcal R$ is well-defined on the considered frequency set and $H^D(s)=\mathcal R(s)(I-\sigma G(s))^{-1}G(s)$ is well-defined and passive.
\end{definition}

{\color{black}
\begin{theorem}[OF-PD Condition for Generalized Definition]\label{theo: generalized}
Let $G(s)$ be a scalar transfer function satisfying Assum.~\ref{assump: transfer function}. For a passivity index $\sigma$ and a frequency $\omega$ with $j\omega$ not a pole of either $G$ or $\mathcal R$ and $1-\sigma G(j\omega)\neq0$, the generalized transformed system $H^D(s)=\mathcal R(s)G(s)/(1-\sigma G(s))$ satisfies the PD condition at $\omega$ if and only if
\begin{equation}\label{equ: PR condition generalized}
\begin{aligned}
    \Re(\mathcal{R}(j\omega))\Re(G(j\omega))-\Im(\mathcal{R}(j\omega))\Im(G(j\omega))
    \geq \sigma |G(j\omega)|^2\Re(\mathcal{R}(j\omega)).
 \end{aligned}
\end{equation}
\end{theorem}

\begin{proof}
Write $\mathcal R(j\omega)=a+jb$ and $G(j\omega)=u+jv$. Then
\begin{equation}
\Re\left(\frac{\mathcal R G}{1-\sigma G}\right)
=\frac{au-bv-\sigma a(u^2+v^2)}{|1-\sigma G|^2}.
\end{equation}
Since $|1-\sigma G|^2>0$ by the well-definedness assumption, the PD condition is equivalent to $au-bv\geq\sigma a(u^2+v^2)$, which is \eqref{equ: PR condition generalized}.
\end{proof}}

Some extended passivity definitions employed in power system analysis are provided as follows

\begin{example}\label{example: 1}
    Let $\mathcal{R}=s$ as the differential passivity used in distributed stability analysis \cite{Yang_Distributed_2020}, also recognized as negative imaginariness \cite{Xiong_Negative_2010, Bhowmick_LTI_2017}. Then, the condition \eqref{equ: PR condition generalized} is equivalent to $\Im(G)\leq 0$ for $\forall\omega\geq 0$, 
    which corresponds to the lower half plane.
\end{example}

\begin{example}\label{example: 2}
    Let $\mathcal{R}=1$ and $sG(s)$ replacing $G(s)$ used in power grid code generation \cite{Peng_Compositional_2025}. Then, the condition \eqref{equ: PR condition generalized}
    represents a disk with $\partial \mathcal{D}$ centered on the imaginary axis for a given $\omega$ when $\sigma>0$.
\end{example}


\begin{example}\label{example: 3}
    Let $\mathcal{R}=s/(\varepsilon s+1)$ as the washout loop \cite{Dey_PassivityBased_2023a}. Then, the condition \eqref{equ: PR condition generalized}
    represents a disk where $\partial \mathcal{D}$ passes through $(0,j0)$ and $(1/\sigma,j0)$ with the imaginary part of its center controlled by $\varepsilon \omega$ when $\sigma>0$.
\end{example}


To deal with possibly existing explicit $\omega$ in $\mathcal{P}_{\rm PD}(\omega)$, consider a three-dimensional coordinate $(\Re(G(j\omega)),\Im(G(j\omega)),\omega)$. Let $\mathcal R(j\omega)=a(\omega)+jb(\omega)$ and $G(j\omega)=u+jv$. The generalized PD region is determined by
\begin{equation}
    a u-b v\geq \sigma a(u^2+v^2).
\end{equation}
When $\sigma a>0$, this is the interior of the disk
\begin{equation}
\left(u-\frac{1}{2\sigma}\right)^2+
\left(v+\frac{b}{2\sigma a}\right)^2
\leq
\frac{1}{4\sigma^2}\left(1+\frac{b^2}{a^2}\right).
\end{equation}
When $\sigma a<0$, the admissible region is the exterior of the same disk. When $a=0$, the condition degenerates to the half-plane $-bv\geq0$; when $\sigma=0$, it degenerates to the half-plane $au-bv\geq0$. Therefore, the closed-form disk expression must be interpreted with this sign classification.
Similarly, we derive the following results for the PD frequency bands.

\begin{corollary}
Consider a scalar transfer function $G(s)$ satisfying Assum.~\ref{assump: transfer function} w.r.t. a passivity index $\sigma$ and define its PD frequency bands $\Omega_{\rm{PD}}^D(\sigma)= \{\omega\geq0: H^D(j\omega) + (H^D)^H(j\omega) \succeq 0\}$.
Then, $\Omega^D_{\rm PD}(\sigma)=\left\{\omega\geq0\right.$: the condition \eqref{equ: PR condition generalized} is satisfied.$\left.\right\}$
\end{corollary}

\section{Encoding Passivization into Classic Tools}
The method's simplicity allows it to be encoded into classic analysis and design tools.
In various applications, the configuration illustrated in Fig.~\ref{figure: Passivity_Transformation} comprises an uncontrollable $G_2$ and a controllable $G_1$. Besides, the OF passivity index $\sigma$ for $G_1$ is commonly predetermined or has a minimum threshold to guarantee the closed-loop stability. 

For example, in the stability analysis and stability-aware grid code designs of power systems \cite{Yang_Distributed_2020, Peng_Compositional_2025, Haberle_Dynamic_2024}, 
$G_2$ represents the network with an IF passivity index $\sigma_2<0$, which is usually challenging to regulate. 

The proposed method asserts that, to maintain closed-loop stability, we desire the device dynamics $G_1$ to lie in the PD region w.r.t. $\sigma_1>-\sigma_2$. Notably, $G_1$ does not need to follow a specific structure or even admit an analytical model. Relaxing these modeling requirements increases flexibility in analysis and controls.

\subsection{Graphic Analysis on Nyquist and Nichols Plots}
The above analysis presents the proposed condition in Nyquist plots with the $\Re(G)$-$\Im(G)$ coordinates.
Apart from that, it is also easily applicable to the Nichols plot with $\log|G|$-$\angle G$ coordinates. Specifically, \eqref{equ: PR condition SISO} equals to $20\log|G(j\omega)|\leq 20\log|\cos(\angle G)|-20\log(\sigma)$
for $\{\omega: |\angle G(j\omega)|\leq\pi/2\}$.
The Nichols plot displays the system's dynamic performance, such as gain margin, phase margin, resonant peak, and bandwidth under unit negative feedback.
Thus, it holds significance in circuit and aircraft controls \cite{Dorf_Modern_2022}. An example is given in Fig.~\ref{figure: Illu_2_Ctrl}.

\begin{figure}[htbp]
    \centering
    \includegraphics[width=3.2in]{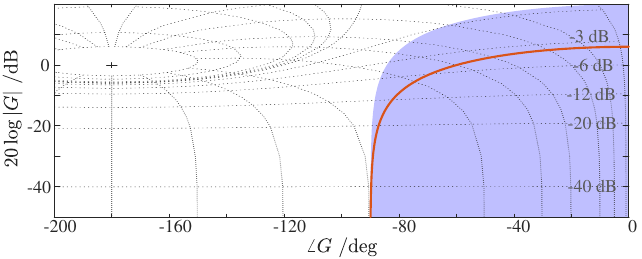}
    \caption{Encoding with Nichols plot. The orange curve is the Nichols plot of $G_1$ in \eqref{equ: case TF}, and the PD region $\mathcal{P}_{\rm PD}(0.1)$ is the purple region.}
 \label{figure: Illu_2_Ctrl}
\end{figure}
Utilizing the Nichols plot, control designs can be performed by deploying the classic frequency-response shaping method. 
The proposed method does not alter the fundamental approach of frequency response shaping; rather, it introduces an additional performance requirement wherein the trajectory must stay within the PD region (e.g., the purple region in Fig.~\ref{figure: Illu_2_Ctrl}), thereby ensuring that $G_1$ is OF passive w.r.t. the index $\sigma_1$. 
We refer interested readers to \cite[Chapter 9]{Dorf_Modern_2022} for further details on the specific design and tuning procedures.

\subsection{Generalized-KYP-Based Parameter Tuning}\label{subsec: control tuning theory}

Although full-frequency passivity may be unattainable in some practical systems, adequate passivity over finite, especially low-frequency, bands is often sufficient in practice. The proposed method translates this engineering requirement into LMI constraints through the generalized KYP lemma \cite{Iwasaki_Generalized_2005}, thereby providing a systematic synthesis method that is consistent with the proposed theory-to-practice perspective.


Specifically, consider a plant $P(s)$ with the minimal realization $(A,B,C,D)$ and a PI controller
\begin{equation}\label{equ: leaky-PI}
    K(s)=k_p+{k_i}/{(s+\varepsilon)}
\end{equation}
In this formulation, the leaky integrator $k_i/(s+\varepsilon)$ is employed to avoid negative damping of the pure integrator.
The PI parameters awaiting tuning can be expressed as the affine vector $\rho=[k_p, k_i]^T$.
Then, it is desired that $G(s)=P(s)K(s)$ satisfies the finite-frequency OF-PD condition over the predetermined low-frequency band $|\omega|\leq\omega_c$ w.r.t. an index $\sigma>0$, i.e.,
\begin{equation}\label{equ: control target}
    G(j\omega)\in\mathcal{P}_{\rm PD}(\sigma), \forall -\omega_c\leq\omega\leq \omega_c
\end{equation}

First, the PD condition $G(j\omega)\in\mathcal{D}(1/(2\sigma),1/(2\sigma))$ can be equivalently expressed as \eqref{equ: passivity LMI}. 
Second, the frequency restriction $|\omega|\leq\omega_c$ can be reformulated as
\begin{equation}
    \Lambda(\Phi, \Psi) = \left\{ \lambda \in \mathbb{C} : \begin{bmatrix} \lambda \\ 1 \end{bmatrix}^* \Phi \begin{bmatrix} \lambda \\ 1 \end{bmatrix} = 0, \begin{bmatrix} \lambda \\ 1 \end{bmatrix}^* \Psi \begin{bmatrix} \lambda \\ 1 \end{bmatrix} \ge 0 \right\}
\end{equation}
where
\begin{equation}
    \Phi=\begin{bmatrix} 0 & 1 \\ 1 & 0 \end{bmatrix},
    \Psi = \begin{bmatrix} -1 & 0 \\ 0 & \omega_c^2 \end{bmatrix}
\end{equation}
Then, the equivalent open-loop system $G(s)=P(s)K(s)$ has the minimal realization $(A_L,B_L,C_L,D_L)$ with the following state matrices
\begin{equation}
\begin{aligned}
    &A_L=\begin{bmatrix}
        -\varepsilon & 0\\
        B & A
    \end{bmatrix}
    &B_L=\mathcal{M}_B\rho=\begin{bmatrix}
        0 & 1\\
        B & 0
    \end{bmatrix}\rho\\
    &C_L=\begin{bmatrix}
        D & C
    \end{bmatrix}
    &D_L =\mathcal{M}_D\rho= \begin{bmatrix}
        D & 0
    \end{bmatrix}\rho
\end{aligned}
\end{equation}

Compute the following coefficient matrix
\begin{equation}
\begin{aligned}
    W &= \begin{bmatrix} A_L & I \\ C_L & 0 \end{bmatrix} (\Phi \otimes P + \Psi \otimes Q) \begin{bmatrix} A_L & I \\ C_L & 0 \end{bmatrix}^T\\
    &= \begin{bmatrix} A_L P + P A_L^T - A_L Q A_L^T + \omega_c^2 Q & P C_L^T - A_L Q C_L^T \\ C_L P - C_L Q A_L^T & -C_L Q C_L^T \end{bmatrix}
\end{aligned}
\end{equation}
\begin{equation}
    \begin{aligned}
        V = \begin{bmatrix} 0 & B_L\Pi_{12} \\ \Pi_{12}^T B_L^T & D_L\Pi_{12} + \Pi_{12}^T D_L^T + \Pi_{22} \end{bmatrix}=\frac{-1}{2\sigma}\begin{bmatrix} 0 & B_L \\ B_L^T & 2D_L\end{bmatrix}
    \end{aligned}
\end{equation}
\begin{equation}
    T = \begin{bmatrix}
        B_L^T & D_L^T
    \end{bmatrix}^T
\end{equation}

The submatrix $\Pi_{11}=1>0$ satisfies \cite[Assumption 1]{Hara_Robust_2006}.
Then, by \cite[Proposition 1]{Hara_Robust_2006}, the target \eqref{equ: control target} can be achieved if there exist Hermitian matrices $P$ and $Q$ and the vector $\rho\geq0$ such that the following LMI \eqref{equ: LMI} holds.
\begin{equation}\label{equ: LMI}
    \Xi=
    \begin{bmatrix}
        W+V & T\\
        T^T & -1
    \end{bmatrix}=
    \begin{bmatrix} \Xi_{11} & \Xi_{12} & \Xi_{13} \\ \Xi_{12}^T & \Xi_{22} & \Xi_{23} \\ \Xi_{13}^T & \Xi_{23}^T & -1 \end{bmatrix} \preceq 0,\; Q\succeq0
\end{equation}
where the entry-wise expressions are
\begin{equation}
    \begin{aligned}
        &\Xi_{11} = A_L P + P A_L^T - A_L Q A_L^T + \omega_c^2 Q\\
        &\Xi_{12} = P C_L^T - A_L Q C_L^T - \frac{1}{2\sigma} \mathcal{M}_B \rho\\
        &\Xi_{22} = -C_L Q C_L^T - \frac{1}{\sigma} \mathcal{M}_D \rho, \; \Xi_{13} = \mathcal{M}_B \rho, \; \Xi_{23} = \mathcal{M}_D \rho
    \end{aligned}
\end{equation}
The above derivations can be summarized as follows.
\begin{corollary}[Passivity-Aware Parameter Tuning]\label{corollary: control LMI}
    Consider a minimal-realized plant $P(s)$ with a PI controller $K(s)$ in \eqref{equ: leaky-PI}.
    Then, the equivalent open-loop transfer function $G(s)=P(s)K(s)$ satisfies \eqref{equ: control target} if there exist Hermitian matrices $P$ and $Q$ and the vector $\rho=[k_p,k_i]^T\geq0$ for the LMI \eqref{equ: LMI}.
\end{corollary}

An example from power inverter control is given in Section \ref{sec: application: control}.
The method extends to any controller with an affine parameter structure \cite{Iwasaki_Generalized_2005}. 
By encoding this constraint, automatic parameter tuning can be performed to render the system OF passive.

\section{Application I: Visualized Passivity Analysis}
This section illustrates the application of the proposed method in passivity analysis.
Consider the transfer functions \eqref{equ: case TF} from electrical power system analysis.
Here, $G_1, G_2, G_3$ and $G_4$ represent the voltage dynamics \cite{Yang_Distributed_2020}, phase-angle dynamics \cite{Yang_Distributed_2020}, frequency swing dynamics with power filter\cite{Yang_Distributed_2020}, and angle-voltage coupled control strategy with cross-loop gain $G_c=C/(Ts+1)$ \cite{Peng_Impact_2024}, respectively. 
$G_5$ represents the controllers derived from time-domain grid codes \cite{Haberle_Dynamic_2024}.
\begin{equation}\label{equ: case TF}
\begin{aligned}
    &G_1=\frac{1}{\tau s+k},\qquad\qquad\qquad\quad
    G_2=\frac{1}{s(Ms+d)}\\
    &G_3=\frac{1}{(Ts+1)(Ms+d)},\qquad
    G_4=\begin{bmatrix}
        G_3 & G_c\\
        G_c & G_1
    \end{bmatrix}\\
    &G_5=\frac{\alpha}{s}\left(\frac{1}{s}
        -(t^r+\frac{1}{s})\frac{1-\frac{t^r}{2}s}{1+\frac{t^r}{2}s}
        +t^r\frac{1-\frac{t^r}{2}s}{1+\frac{t^r}{2}s}
        -t^r\frac{1-\frac{t^e}{2}s}{1+\frac{t^e}{2}s}\right)
\end{aligned}  
\end{equation}

\subsection{Graphic Analysis of Classical Passivity}


Fig.~\ref{figure: Simulation1_SISO}(a) shows that $G_1$ is not OF-passivizable for $\sigma=1$ since its Nyquist plot stays outside $\mathcal{P}_{\rm PD}(1)$ (red disk). Ensuring $G_1(j\omega)\in\mathcal{P}_{\rm PD}(\sigma)$ for all $\omega$ requires the disk diameter $1/\sigma$ to exceed the rightmost Nyquist value $1/k=2$.
Here, we select $1/\sigma = 3$ (purple disk), which makes $G_1$ OF passive.

\begin{figure}[htbp]
    \centering
    \includegraphics[width=3.4in]{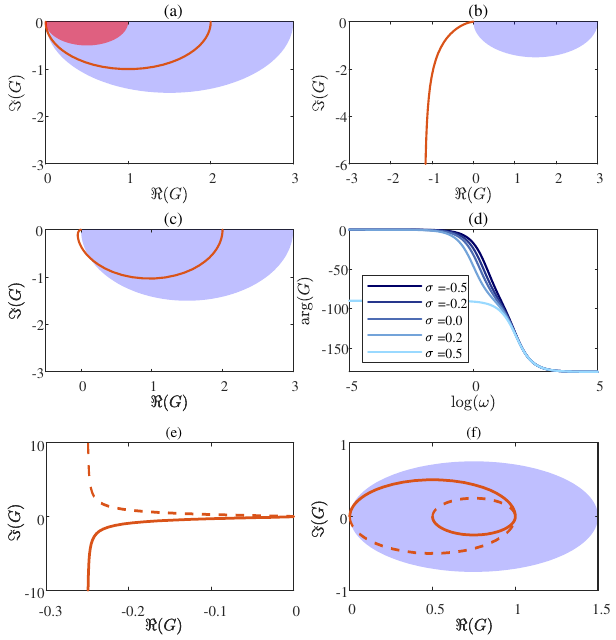}
    \caption{Analysis of SISO cases. (a): Nyquist plot (orange curve) of $G_1$ and the PD region $\mathcal{P}_{\rm PD}(1)$ (red disk) and $\mathcal{P}_{\rm PD}(1/3)$ (purple disk). (b): Nyquist plot of $G_2$. (c): Nyquist plot of $G_3$. (d): Bode plot of $G_3$. (e): Nyquist plot of $G_5$. (f): Nyquist plot of $sG_5$.}
 \label{figure: Simulation1_SISO}
\end{figure}

Fig.~\ref{figure: Simulation1_SISO}(b) shows that the Nyquist plot of $G_2$ cannot intersect any disk w.r.t. $\sigma \geq 0$, implying that $G_2$ cannot be passivized via OF for any $\sigma \geq 0$. 
In other words, the synchronous generator dynamics are not passive.
Although this conclusion can be derived analytically, the proposed geometric approach offers a more intuitive and quicker check.

The Nyquist plot of $G_3$ is shown in Fig.~\ref{figure: Simulation1_SISO}(c). It can be visually confirmed that $G_3$ cannot be fully passivized across all frequencies, especially the high-frequency bands. For the purple disk w.r.t. $\sigma = 1/3$, the system meets the PD condition only when the Nyquist curve lies inside the disk, corresponding to the critical frequency $\omega_c\leq 5.3703$. Moreover, as $\sigma$ increases, the PD frequency band monotonically contracts as analyzed in previous sections. As illustrated in Fig.~\ref{figure: Simulation1_SISO}(d), for $\sigma =$ -0.5, -0.2, 0, 0.2, 0.5, the critical frequencies are 13.1826, 10.9648, 9.3325, 7.0795, and 0.0000, respectively.

The passivization under the generalized definition in Ex.~\ref{example: 2} can be visualized using the classical PD-region result.
Although $G_5$ is non-passive under the classical definition in Fig.~\ref{figure: Simulation1_SISO}(e), the inclusion 
$j\omega G_5(j\omega)\in\mathcal{P}_{\rm PD}(2/3)$ in Fig.~\ref{figure: Simulation1_SISO}(f) indicates passivity in the sense of Ex.~\ref{example: 2}.
It also illustrates the advantages of the proposed method: although the expression of $G_5$ in \eqref{equ: case TF} is quite complex, the graphic analysis remains intuitive.

\begin{figure}[htbp]
    \centering
    \includegraphics[width=3.4in]{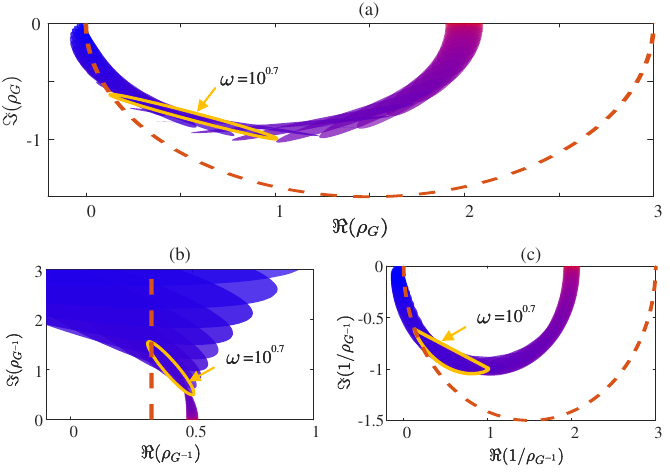}
    \caption{Analysis of MIMO cases. (a): $\rho_{G(j\omega)}$. (b): $\rho_{G^{-1}(j\omega)}$ with the PD region \eqref{equ: MIMO rayleigh equivalent}. (c): $\rho_{G^{-1}(j\omega)}$ with the PD region \eqref{equ: MIMO rayleigh equivalent-inverse}. 
    The orange dashed circle denotes the PD region $\mathcal{P}_{\rm PD}(1/3)$.
    Each colored disk represents the Rayleigh quotient $\rho_G(j\omega)$ at a given frequency $\omega$. 
    For example, the yellow circle represents the boundary of the $\rho_{G(j10^{0.7})}(x)$, which intersects the PD region. Here we present the frequencies with  $\log(\omega)$ from $-3.0:0.1:2.0$.}
 \label{figure: Simulation1_MIMO}
\end{figure}

The analysis remains applicable to MIMO systems. Fig.~\ref{figure: Simulation1_MIMO}(a) illustrates the Rayleigh quotient of $\rho_{G}$ as delineated in $G_4(j\omega)$  and the PD region w.r.t. the passivity index $\sigma=1/3$. According to \eqref{equ: MIMO LMI expression} of Thm.~\ref{theo: MIMO}, $G_4$ cannot satisfy the PD condition and thus cannot be OF passivized in the high-frequency band $\omega\geq 10^{0.7}$ rad/s, where the Rayleigh quotient cannot reside within the PD region.
In Fig.~\ref{figure: Simulation1_MIMO}(b)-(c), the analysis of $\rho_{G^{-1}}$ aligns with the above results based on \eqref{equ: MIMO rayleigh equivalent} and \eqref{equ: MIMO rayleigh equivalent-inverse} of Thm.~\ref{theo: MIMO}.

\subsection{Graphic Analysis of Generalized Passivity}

Fig.~\ref{figure: Simulation2} illustrates the PD regions of $G_2$ and $G_3$ under the differential passivity definition in Ex.~\ref{example: 1}. For any negative imaginary transfer function, the differential passivity meets the PD condition for the OF system, which also explains the advantageous applicability of this definition in power systems \cite{Yang_Distributed_2020}. However, it should be noted that this does not imply that an arbitrary $\sigma$ can render the transfer function passive. This is because, in addition to the PD condition, passivity must also satisfy conditions 1) and 3) in Def.~\ref{def: passivity}, which can be guaranteed geometrically in a similar approach to Thm.~\ref{theo: SISO passivity condition}. For $G_2$, to ensure that its Nyquist plot, as shown in Fig.~\ref{figure: Simulation2}(a), does not encircle the point $(1/\sigma, j0)$, it is necessary to choose $1/\sigma<0$. For $G_3$, to prevent its Nyquist plot in Fig.~\ref{figure: Simulation2}(b) from encircling the point $(1/\sigma, j0)$, the condition $1/\sigma > 1/d$ must be satisfied, i.e., $\sigma < d$, which aligns with \cite{Yang_Distributed_2020, Peng_Compositional_2025}.

\begin{figure}[htbp]
    \centering
    \includegraphics[width=3.4in]{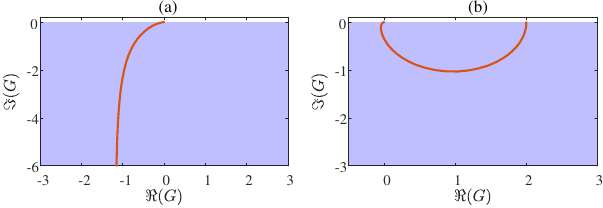}
    \caption{Analysis of differential passivization in Ex.~\ref{example: 1}. (a): Nyquist plot (orange) of $G_2$ with PD region (purple). (b): Nyquist plot of $G_3$ with PD region.}
 \label{figure: Simulation2}
\end{figure}


The PD region for the generalized passivity in Ex.~\ref{example: 2} is presented in Fig.~\ref{figure: Simulation3_EX2}. The case for $\sigma = 0.1$ is shown in Fig.~\ref{figure: Simulation3_EX2}(a). The Nyquist plots of all three transfer functions $G_1$, $G_2$, and $G_3$ lie entirely within the PD region, indicating that the corresponding systems are OF passivizable in the frequency band of interest.
Furthermore, we examine the cases for different values of $\sigma$. The corresponding PD regions for $\sigma = 0.1$, $0.3$, and $1.0$ are displayed in Fig.~\ref{figure: Simulation3_EX2}(b). Firstly, it can be observed that the region contracts as $\sigma$ increases. Secondly, $G_2$ has a positive damping only for $\sigma = 0.1$ and $0.3$, while certain frequency bands fall outside the region w.r.t. $\sigma = 1.0$ and thus are non-passivizable.

\begin{figure}[htbp]
    \centering
    \includegraphics[width=3.2in]{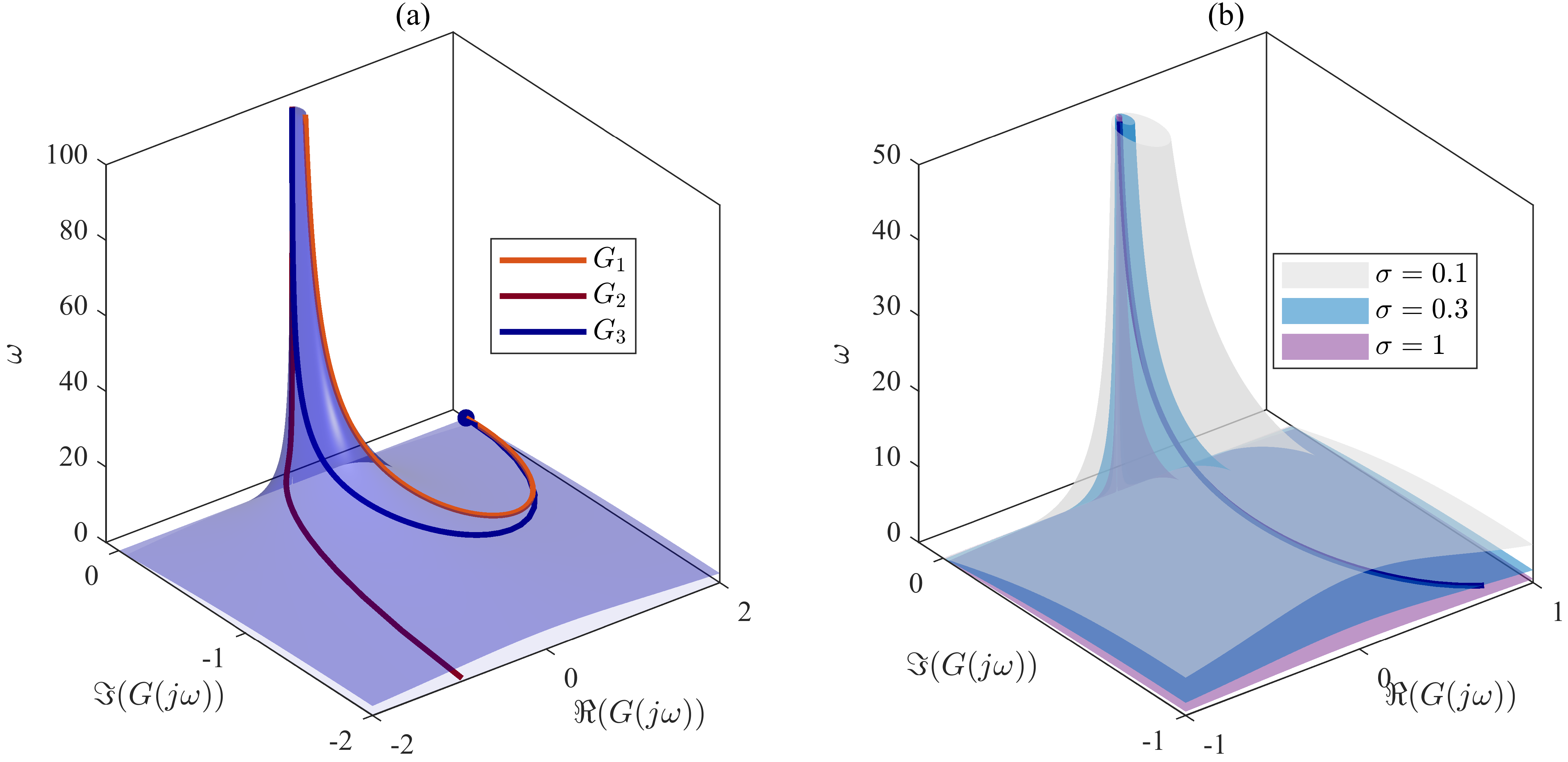}
    \caption{Analysis of Ex.~\ref{example: 2} passivization. (a): PD region w.r.t. $\sigma=$0.1 (purple region) and the Nyquist plots of $G_1$, $G_2$, and $G_3$. (b): PD regions w.r.t. $\sigma=$0.1, 0.3 and 1.0 and the Nyquist plot of $G_3$.}
 \label{figure: Simulation3_EX2}
\end{figure}

Different definitions of passivity lead to distinct PD regions. Fig.~\ref{figure: Simulation3_EX3} illustrates the PD regions corresponding to Ex.~\ref{example: 3}. Then, systems $G_1, G_2$, and $G_3$ remain passivizable, as illustrated in Fig.~\ref{figure: Simulation3_EX3}(a). Despite this, the PD region diverges from that in Ex.~\ref{example: 2}. Notably, Ex.~\ref{example: 2}'s PD region reduces to the origin as $\omega\to\infty$, while that of Ex.~\ref{example: 3} converges to $\mathcal{D}(1/2+j0, 1/2)$.

\begin{figure}[htbp]
    \centering    \includegraphics[width=3.in]{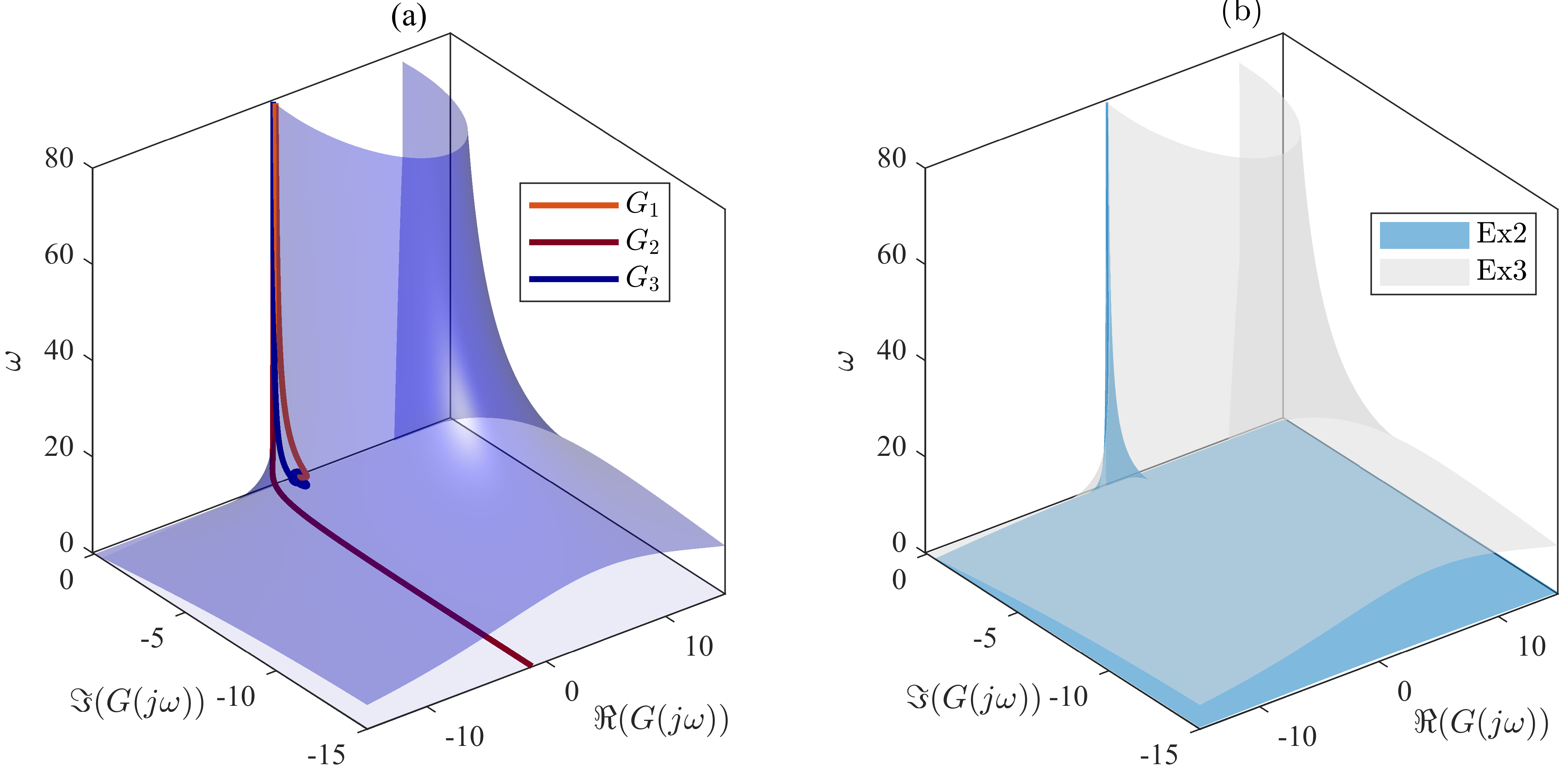}
    \caption{Analysis of  Ex.~\ref{example: 3} passivization. 
    (a): PD region w.r.t. $\sigma=0.1$ (purple region) and the Nyquist plots of $G_1$, $G_2$ and $G_3$. (b): comparison of PD regions of Ex.~\ref{example: 2} and \ref{example: 3} w.r.t. $\sigma=0.1$. }
 \label{figure: Simulation3_EX3}
\end{figure}

\section{Application II: Passivity-Aware Controls}\label{sec: application: control}
\subsection{Passivity-Aware Control Design Method}
This section applies the proposed methods to passivity-aware control design.
Consider the voltage control scheme of grid-forming inverters in Fig.~\ref{figure: inverter}.
Detailed demonstration of inverter models can be found in classical books 
and also our attached source code.
A common strategy is to set the fixed-voltage reference to a static $K(s)=0$ in Fig.~\ref{figure: inverter}. 
Its dynamics can be approximated as \cite{Ravanji_Impact_2023}: 
\begin{equation}
\begin{aligned}
    G_{v}(s) &= \frac{G_i(k_{vp}s+k_{vi})}{C_fs^2+G_i(k_{vp}s+k_{vi})}\\
    G_i(s) &= \frac{k_{ip}s+k_{ii}}{L_fs^2+(k_{ip}+R_f)s+k_{ii}}
\end{aligned}
\end{equation}

Although simple to implement, the fixed-reference control might become unstable under stiff grids, as shown in Fig.~\ref{figure: inverter_transient}(a).
Therefore, a dynamic $K(s)$ is desired to enhance stability. 
Now, $G_v(s)$ serves as the plant awaiting the reshaping of its frequency response.
We particularly consider a PI structure \eqref{equ: leaky-PI} named power synchronization control.

\begin{figure}[htbp]
	\centering
	\includegraphics[width=2.8in]{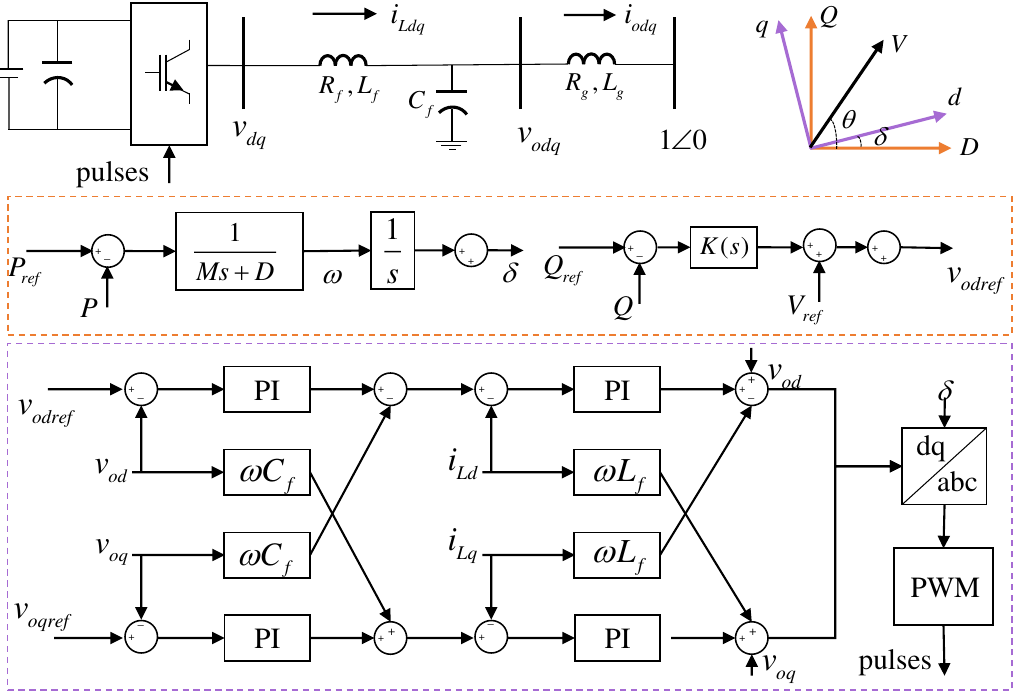}
	\caption{Grid-forming inverter model.}
 \label{figure: inverter}
\end{figure}

Now, our target is to tune the PI parameters $(k_p, k_i)$ in \eqref{equ: leaky-PI} to achieve the predetermined passivity performance through the proposed method. Specifically, the inverters are desired to satisfy a finite-frequency OF-PD index no smaller than $\sigma$ over the low-frequency band $-\omega_c\leq\omega\leq\omega_c$.
By employing Cor.~\ref{corollary: control LMI}, it can be transformed into solving the LMI \eqref{equ: LMI}.
A solution is $(k_p,k_i)=(0.0555, 0.0345)$. In the frequency domain, its Nyquist plot remains in the PD region under the given frequency, as shown in Fig.~\ref{figure: inverter_freq_analysis}(a). In electromagnetic transient simulations, it is also stabilized as depicted in Fig.~\ref{figure: inverter_transient}(a).
Moreover, the feasible $(k_p,k_i)$ satisfying the passivity requirement forms a region in Fig.~\ref{figure: inverter_freq_analysis}(b), which provides more degrees of freedom for other control targets.

\begin{figure}[htbp]
	\centering
	\includegraphics[width=3.1in]{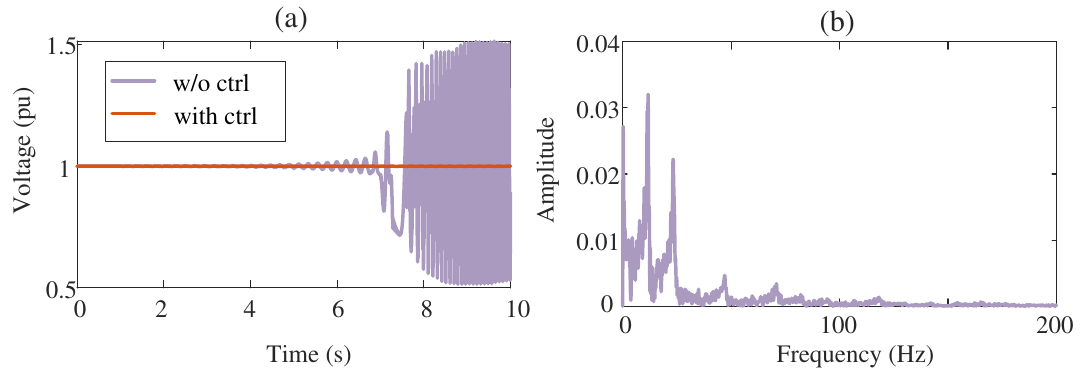}
	\caption{Transient simulation of grid-forming inverters with passivity-aware control designs. (a): voltage transients of inverters without or with dynamic controllers. A +1\% step change in the active-power reference $P_{\rm ref}$ is applied as the disturbance. (b): Fourier analysis of inverter voltage transients without dynamic controllers.}
 \label{figure: inverter_transient}
\end{figure}

\begin{figure}[htbp]
	\centering
	\includegraphics[width=3.1in]{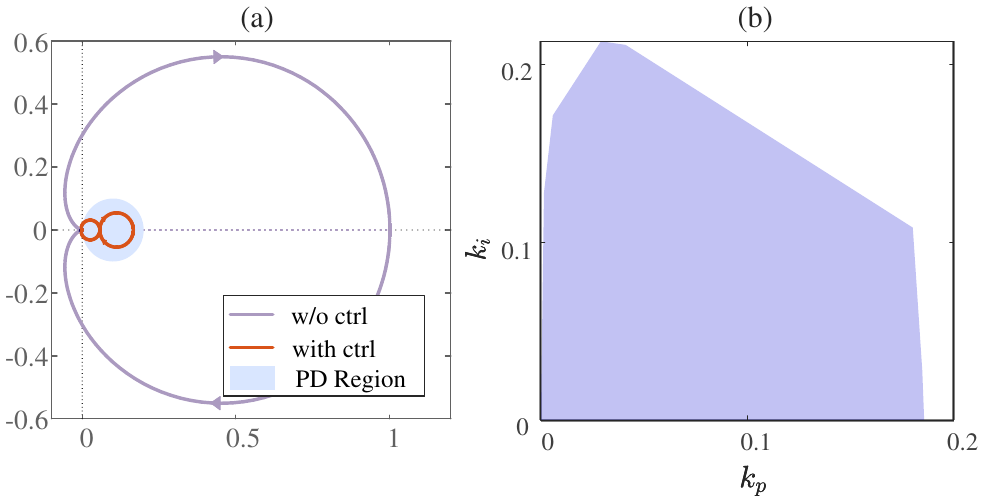}
	\caption{Passivity-aware control design for inverter voltage dynamics. (a): Nyquist plots without or with dynamic controllers. (b): feasible parameter region $(k_p, k_i)$ under given $\sigma=5$ and $\omega_c=100\cdot 2\pi$ rad/s.}
 \label{figure: inverter_freq_analysis}
\end{figure}

\subsection{Implementation Issues}
The following issues are worth noting in implementations:
\begin{itemize}
    \item Determination of targeted frequency bands: The cutoff frequency $\omega_c$ should cover the typical oscillation frequencies. In our example, the fixed-reference inverters mainly suffer from the oscillations under 100 Hz in Fig.~\ref{figure: inverter_transient}(b). Therefore, $\omega_c=100\cdot 2\pi$ rad/s is selected. 
    \item Determination of passivity index: The selection of the passivity index is critical. In power system distributed stability\cite{Peng_Compositional_2025}, the OF-passivity $\sigma_1$ of devices should exceed the IF-passivity shortage $\sigma_2$ of networks (see Fig.~\ref{figure: Passivity_Transformation}). Therefore, $\sigma_1>-\sigma_2$ with sufficient margins is preferred.
    \item Numerical robustness improvement: In practical implementations, the LMI \eqref{equ: LMI} is actually not solved with parameters $(A_L, B_L, C_L,D_L,\omega_c)$, but with $(A_L/\varpi, B_L/\sqrt{\varpi}, C_L/\sqrt{\varpi},D_L,\omega_c/\varpi)$, where $\varpi$ is the frequency scaling factor. Both results are identical since
    \begin{equation}
    \begin{aligned}
        L(s) &= C_L(sI - A_L)^{-1}B_L + D_L\\
        &= \left( \frac{C_L}{\sqrt{\varpi}} \right) \left( \frac{s}{\varpi} I - \frac{A_L}{\varpi} \right)^{-1} \left( \frac{B_L}{\sqrt{\varpi}} \right) + D_L
    \end{aligned}
    \end{equation}
    This transformation is performed to reduce numerical ill-conditions induced by the distinct order of magnitude between $\Xi_{11}$
    and $\Xi_{33}$.
\end{itemize}

\section{Conclusion}

This work develops a graphical framework that bridges passivization theory and engineering practice. By showing that the PD conditions associated with IF and OF passivization can be assessed exactly in SISO cases, exactly by LMI/inverse numerical-range tests in MIMO cases, and diagnostically by Rayleigh-quotient regions, the proposed method converts abstract frequency-domain conditions into a directly interpretable tool. Moreover, the framework interfaces naturally with classical analysis and design methods, including Nichols-plot-based reasoning and generalized KYP-based controller synthesis.

This perspective not only clarifies fundamental trade-offs, such as the contraction of the positive damping bandwidth as the passivity index increases, but also provides a unified basis for both classical and generalized passivity analysis. The case studies in power systems show how the proposed framework can support practical assessment, stability analysis, and controller tuning.

Future work will focus on the generalization to nonlinear systems. In particular, the extended passivity definition introduced in Sec~\ref{subsec: extension to generalized passivity definition} aligns with the Zames-Falb methods for analysis of Lur'e systems \cite{carrasco2019convex}.
Therefore, the proposed method can be applied to Lur\'e-type nonlinear systems.